\documentclass[12pt,english,final]{amsart}
\usepackage[utf8]{inputenc}
\usepackage{geometry}
\usepackage{amstext}
\usepackage{amsthm}
\usepackage{amssymb}

\makeatletter
\numberwithin{equation}{section}
\numberwithin{figure}{section}
\theoremstyle{plain}
\newtheorem{thm}{\protect\theoremname}
\theoremstyle{remark}
\newtheorem{rem}[thm]{\protect\remarkname}
\theoremstyle{plain}
\newtheorem{prop}[thm]{\protect\propositionname}
\theoremstyle{plain}
\newtheorem{lem}[thm]{\protect\lemmaname}
\theoremstyle{plain}
\newtheorem{cor}[thm]{\protect\corollaryname}

\usepackage{amsfonts}\usepackage{exscale}\usepackage{cite}\usepackage{epsfig}\usepackage{amscd}\usepackage{graphics}\usepackage{color}\usepackage{dsfont}
\usepackage[english]{babel}

\usepackage{bbold}

\usepackage{newunicodechar}
\newcommand{\warning}{%
 \makebox[1.4em][c]{%
 \makebox[0pt][c]{\raisebox{.1em}{\small!}}%
 \makebox[0pt][c]{\color{red}\Large$\bigtriangleup$}}}%

\newunicodechar{⚠}{\warning}

\renewcommand{\geq}{\geqslant}
\renewcommand{\leq}{\leqslant}

\newtheorem*{main-theorem}{Main Theorem}\newtheorem*{theorem*}{Theorem}

\theoremstyle{definition}

\newtheorem*{remark*}{Remark}

\numberwithin{equation}{section}

\def\phi{\varphi}

\def\reals{{\mathbb R}}
\def\cx{{\mathbb C}}

\def\Ci{{\mathcal C}^\infty}

\def\Re{\,\mathrm{Re}\,}
\def\Im{\,\mathrm{Im}\,}

\def\phi{\varphi}

\def\Spec{\operatorname{Spec}\,}
\def\be{\begin{eqnarray*}}
\def\ee{\end{eqnarray*}}
\def\ben{\begin{eqnarray}}
\def\een{\end{eqnarray}}
\def\ker{\text{ker}}
\def\MM{\mathcal{M}}
\def\lll{\left\langle}
\def\rrr{\right\rangle}

\def\L2R{L_{\text{Rest}}^2}

\def\11{\mathds{1}}

\def\HH{\mathcal{H}}

\def\RR{\mathcal{R}}

\def\L2c{L^2_{\text{comp}}}

\def\C{\mathcal{C}}

\def\11{\mathbb{1}}

\def\p{\partial}

\def\Ombar{\overline{\Omega}}
\def\Tr{\text{Tr}}

\usepackage{babel}
\providecommand{\corollaryname}{Corollary}
\providecommand{\lemmaname}{Lemma}
\providecommand{\propositionname}{Proposition}
\providecommand{\remarkname}{Remark}
\providecommand{\theoremname}{Theorem}

\makeatother

\usepackage{babel}
\providecommand{\corollaryname}{Corollary}
\providecommand{\lemmaname}{Lemma}
\providecommand{\propositionname}{Proposition}
\providecommand{\remarkname}{Remark}
\providecommand{\theoremname}{Theorem}

\begin{document}
\global\long\def\scal#1#2{\langle#1,#2\rangle}%

\title{Wave Decay with Singular Damping}
\author{Hans Christianson}
\author{Emmanuel Schenck}
\author{Michael Taylor}
\begin{abstract}
We consider the stabilization problem on a manifold with boundary
for a wave equation with measure-valued linear damping. For a wide
class of measures, containing Dirac masses on hypersurfaces as well
as measures with fractal support, we establish an abstract energy decay
result. 
\end{abstract}

\maketitle

\section{Introduction}

In this paper, we consider a damped wave equation with measure-valued
damping on a compact Riemannian manifold $(\Omega,g)$ with non-empty
smooth boundary. If the measure is the Dirac mass on a hypersurface,
then this problem has been considered on an interval in \cite{BRT-D-damp,JTZ-D-damp}
and on bounded domains in $\reals^{2}$, provided the hypersurface
stays away from the boundary. In particular, in \cite{JTZ-D-damp},
it is shown that generically for curved hypersurfaces, the energy
of corresponding solutions decays to $0$. However, if the domain
is strictly convex, the existence of whispering gallery modes highly
concentrated along the boundary \cite{Ral69_loc} shows that the decay
rate is not uniform.

In this paper, we develop a functional analysis approach, assuming
that the measure has good mapping properties between (generalized)
Sobolev spaces. The novelty here is that the class of measures considered
here is much wider than restrictions of the Riemannian volume to hypersurfaces
(see Section \ref{S:compact-measure-mapping}). This mapping assumption
allows us to prove an abstract energy decay result rather directly,
in any dimension $n\geq2$. 

Let $-\Delta_{g}$ be the (positive) Laplace-Beltrami operator on
$\Omega$. Let $dV_{g}$ denote the Riemannian volume element, and
let $\lll\cdot,\cdot\rrr_{L^{2}(\Omega)}$ denote the $L^{2}$ inner
product with respect to $dV_{g}$, with the convention that it is
linear in the first argument. Since the boundary of $\Omega$ is non-empty,
we prescribe Dirichlet boundary conditions to $-\Delta_{g}$, so that
the domain is 
\[
D(-\Delta_{g})=H^{2}(\Omega)\cap H_{0}^{1}(\Omega).
\]
Here $H_{0}^{1}(\Omega)$ is the completion of $\Ci_{c}(\Omega)$
with respect to the homogeneous norm 
\[
\|u\|_{H_{0}^{1}(\Omega)}=\lll\nabla u,\nabla u\rrr_{L^{2}(\Omega)}^{1/2}.
\]
We observe that the Dirichlet boundary conditions imply that for $u\in\Ci_{c}(\Omega)$,
\begin{align*}
0 & =\|u\|_{H_{0}^{1}(\Omega)}^{2}=\langle\nabla u,\nabla u\rangle_{L^{2}(\Omega)}\Leftrightarrow\ u\text{ is constant,}
\end{align*}
so this is a norm. In fact, we have 
\begin{align*}
\|u\|_{H_{0}^{1}(\Omega)}^{2} & =\langle-\Delta u,u\rangle_{L^{2}(\Omega)}\geq\lambda_{0}\|u\|_{L^{2}(\Omega)}^{2},
\end{align*}
where $\lambda_{0}>0$ is the first Dirichlet eigenvalue (necessarily
positive).

Let $M:\mathcal{C}_{c}\to\mathcal{C}_{c}^{*}$ be a symmetric, non-negative
linear operator which extends to a symmetric, non-negative, bounded
mapping $M:H_{0}^{1}(\Omega)\to H^{-1}(\Omega)$. Such operators will
be the most general ``friction'' coefficients we consider in this
paper. For the functional analysis approach developed here, $M$ need
not be a local operator, but for applications we are most interested
in the case where $M$ is multiplication by a non-negative measure. 

Let $\langle\cdot,\cdot\rangle_{H^{-1}\times H^{1}}$ be the duality
product between $H^{-1}(\Omega)$ and $H^{1}(\Omega)$ so that for
$g\in H^{1}(\Omega)$, 
\[
M(f)(g)=\langle M(f),g\rangle_{H^{-1}\times H^{1}}.
\]
In order to use $M$ in a wave equation with singular friction, we
will write $\langle M(f),g\rangle$ as a shorthand for $\langle M(f),g\rangle=\langle M(f),\overline{g}\rangle_{H^{-1}\times H^{1}}$
. In the case where $\mu$ is a non-negative measure on $\Omega$,
let $M=M_{\mu}$ denote the ``multiplication'' operator $M_{\mu}:\mathcal{C}_{c}\to\mathcal{C}_{c}^{*}$
given by 
\[
M_{\mu}(f)(g)=\int_{\Omega}fgd\mu,\qquad\langle M_{\mu}f,g\rangle=\int_{\Omega}f\bar{g}d\mu.
\]
Through this paper, we will make the following central assumption
: 
\begin{equation}
M:H_{0}^{1}(\Omega)\longrightarrow H^{-1}(\Omega)\text{ is a compact operator.}\label{eq: main assumption}
\end{equation}

Notice no assumption on the support of $\mu$ is made yet, or on the
regularity. For example, $\mu$ need not be absolutely continuous
with respect to $dV_{g}$ : see Section \eqref{S:compact-measure-mapping}
where several examples of such measures are given.

We will be interested in a wave equation with $M$-damping: 
\begin{equation}
\begin{cases}
(\p_{t}^{2}-\Delta+M\p_{t})u=0,\text{ on }(0,\infty)_{t}\times\Omega\\
u|_{\p\Omega}=0,\text{ on }[0,\infty)_{t}\times\p\Omega\\
u(0,x)=g(x),\,\,u_{t}(0,x)=f(x).
\end{cases}\label{E:sing-damp}
\end{equation}
The natural space for the initial data $(f,g)$ is\footnote{In fact, we will take a slightly smaller subspace for our initial
data, namely $(f,g)^{T}\in D(G)$, where $G$ is the matrix operator
defined in \eqref{E:G-def}.} 
\[
\HH=L^{2}(\Omega)\oplus H_{0}^{1}(\Omega).
\]

We will say that $u\in C((0,\infty)_{t})H_{0}^{1}(\Omega)$ is a weak
solution to \eqref{E:sing-damp} if for every $\phi\in H_{0}^{1}(\Omega)$,
\[
\lll\p_{t}^{2}u,\phi\rrr_{L^{2}(\Omega)}+\lll\nabla u,\nabla\phi\rrr_{L^{2}(\Omega)}+\lll M\partial_{t}u,\phi\rrr=0.
\]
As a first result, we prove the following:
\begin{thm}
\label{T:m-damp-1} Suppose the linear operator $M$ extends to a
bounded linear mapping $M:H_{0}^{1}(\Omega)\to H^{-1}(\Omega)$, and
that this mapping is compact. Then $G$ is maximally dissipative,
and $(I-G)^{-1}:\HH\to\HH$ is compact. In particular, $G$ generates
a semigroup of contractions $e^{tG}:\HH\to\HH$, $t\geq0$. 
\end{thm}

\begin{rem}
As the proof will indicate, we will also describe the domain $D(G)$
by describing the domain of $(I-G)$. The strategy is to show that
the real part of $G$ is non-positive definite, then solve the Dirichlet
problem. This will compute $D(G)$, show that $(I-G)$ is invertible
on appropriate Hilbert spaces, and that the inverse is a compact operator.
\end{rem}

The energy for this equation is the usual wave energy: 
\[
E(u,t)=\int_{\Omega}(|u_{t}|^{2}+|\nabla u|^{2})dV_{g}.
\]
Since $M$ is symmetric and non-negative, if $(f,g)^{T}\in D(G)$,
then we compute 
\begin{align*}
E'(u,t) & =2\Re\int_{\Omega}\bar{u_{t}}(u_{tt}-\Delta_{g}u)dV_{g}\\
 & =2\Re\lll u_{tt}-\Delta_{g}u,u_{t}\rrr\\
 & =2\Re\lll-Mu_{t},u_{t}\rrr\\
 & =-2\lll Mu_{t},u_{t}\rrr\\
 & \leq0.
\end{align*}
It is a standard question to determine whether and how this energy
decays to 0 asymptotically in time. As usual, we make the distinction
between strong stabilization, where the energy of every solution to
\eqref{E:sing-damp} goes to 0 as $t\to\infty$ (or, equivalently,
$e^{tG}\xrightarrow{t\to+\infty}0$ strongly in $\HH$) and uniform
stabilization, where there is $f(t)\xrightarrow{t\to+\infty}0$ such
that for all solution $u$ of \eqref{E:sing-damp}, $E(u,t)\leq f(t)$
at all times $t\geq0$. In this latter case, a standard argument shows
that the decay is actually exponential. To determine whether stabilization
holds involves necessarily additional assumptions of geometrical nature:
for instance, it is clear that if the support of $\mu$ is localized
on a nodal set of a Dirichlet eigenfunction of $-\Delta_{g}$, there
are solutions of \eqref{E:sing-damp} with constant energy, and strong
stabilization does not hold. On the other hand, for the standard damped
wave equation -- that is, with a smooth damping term of the form
$au_{t}$ with non-zero $a\in C^{\infty}(\Omega,\reals_{+})$, the
geometric control hypothesis ensures a uniform decay \cite{BLR92} following \cite{RaTa-74} for the case of empty boundary.
In the present case of singular damping, our analysis does not allow
to explore this question further than a spectral condition (Theorem
\ref{thm: Gap} below). However, our main result give a natural condition
for strong stabilization to hold: 
\begin{thm}
\label{thm: stabilization}Let $\mu$ be a finite positive measure
on $\overline{\Omega}$ such that \eqref{eq: main assumption} holds.
For $\lambda\in\Spec(-\Delta_{g})$, denote by $\psi_{\lambda}\in L^{2}(\Omega)$
any eigenfunction with eigenvalue $\lambda$. The following statements
are equivalent :

(i) The semigroup $e^{tG}$ satisfies $e^{tG}\xrightarrow{t\to+\infty}0$
strongly.

(ii) For any $\lambda\in\Spec(-\Delta_{g})$, and any eigenfunction
$\psi_{\lambda}\in L^{2}(\Omega)$ with eigenvalue $\lambda$, the
measure $\mu$ satisfies 

\begin{equation}
\int_{\Omega}|\psi_{\lambda}|^{2}d\mu>0.\label{eq: nodal}
\end{equation}
In other words, the energy of any solution of \eqref{E:sing-damp}
decays to 0 if and only if \eqref{eq: nodal} holds true.
\end{thm}

Note that for a regular damping of the form $\mu=a(x)dV_{g}$ with
non-zero $a\in C^{\infty}(\Omega,\reals_{+})$, condition \emph{(ii)
}is always satisfied since an eigenfunction cannot vanish on an open
set by Holmgren's uniqueness theorem. We recover here a standard result,
namely the strong stabilization for (non-trivial) smooth damping.

\section{Measures that multiply $H_{0}^{1}(\Omega)$ to $H^{-1}(\Omega)$\label{S:compact-measure-mapping}}

In this section, let $M$ be a compact, $n$-dimensional Riemannian
manifold, and $\Omega\subset M$ a connected open subset. Let $\mu$
be a finite, positive measure on $\Ombar$. We want to give conditions
that imply 
\begin{equation}
M_{\mu}:H_{0}^{1}(\Omega)\longrightarrow H^{-1}(\Omega),\ \text{ compactly, where}\quad M_{\mu}f=f\mu.\label{A.1}
\end{equation}
It will actually be convenient to let $\mu$ be a positive finite
measure on $M$, and ask when 
\begin{equation}
M_{\mu}:H^{1}(M)\longrightarrow H^{-1}(M),\ \text{ compactly.}\label{A.2}
\end{equation}

Since $H_{0}^{1}(\Omega)$ is a closed linear subspace of $H^{1}(M)$,
we always get \eqref{A.1} from \eqref{A.2}, by restriction (also
restricting $\mu$ to $\Ombar$). We emphasize that the support of
$\mu$ can have nonempty intersection with $\p\Omega$. 

To start with, let us mention conditions for 
\begin{equation}
M_{\mu}:H^{1}(M)\longrightarrow H^{-1}(M),\label{A.30}
\end{equation}
i.e., \eqref{A.2} but disregarding compactness. If we take $f\in H^{1}(M)$
and pair $M_{\mu}f$ with $f$, since $\mu$ is a positive measure,
a necessary condition for \eqref{A.30} is that 
\begin{equation}
\int\limits _{M}|f|^{2}\,d\mu\le C\|f\|_{H^{1}(M)}^{2},\quad\forall\,f\in H^{1}(M).\label{A.31}
\end{equation}
Applying Cauchy's inequality to $\int_{M}f\overline{g}\,d\mu$, we
see that \eqref{A.31} is also sufficient for \eqref{A.30}. Furthermore,
it follows from Theorem 1.2.2 of \cite{MazSha09} that the existence
of $A<\infty$ such that 
\begin{equation}
\mu(S)\le A\,\text{Cap}(S),\quad\forall\,\text{Borel }S\subset M\label{A.32}
\end{equation}
is necessary and sufficient for \eqref{A.31} to hold, hence for \eqref{A.30}
to hold, where $\text{Cap}(S)$ is a variant of electrostatic capacity,
appropriate for this situation.

A general way to obtain compacity for the map \eqref{A.30} proceeds
as follows. If $\dim M=n\ge2$, standard results for pointwise multiplication
in Sobolev spaces \cite{RunSic96} ensure that 

\begin{align}
f,g & \in H^{1}(M)\Rightarrow fg\in H^{1,p}(M),\quad\text{for }p=\frac{n}{n-1}\ \text{ if }\ n\ge3,\label{eq: mult}\\
f,g & \in H^{1}(M)\Rightarrow fg\in H^{1,p}(M),\quad\text{for all }p<2\ \text{ if }\ n=2.\nonumber 
\end{align}

 Since $C(M)$ is dense in the duals of such spaces $H^{1,p}(M)$,
we deduce from the preceding equations that \eqref{A.2} holds under
the following conditions. 
\begin{align}
\mu & \in H^{1,n/(n-1)}(M)^{*}=H^{-1,n}(M),\ \text{ if }\ n\ge3,\label{A.4}\\
\mu & \in H^{-1,r}(M),\ \text{ for some }\ r>2,\ \text{ if }\ n=2.\nonumber 
\end{align}

\subsection{Measures related to hypersurfaces $S\subset M$}

We can use the above discussion to obtain the following class of singular
measures satisfying \eqref{A.2}.
\begin{prop}
\label{P:A.1} Let a compact $S\subset M$ be locally a Lipschitz
graph, of dimension $n-1$, equipped with surface measure, i.e., $(n-1)$-dimensional
Haudorff measure, $\sigma_{S}$. Let 
\begin{equation}
\mu=h\,\sigma_{S},\label{A.5}
\end{equation}
with 
\begin{align}
h & \in L^{n-1}(S,d\sigma_{S}),\quad\text{if }\ n\ge3,\label{A.6}\\
h & \in L^{1+\delta}(S,d\sigma_{S}),\quad\text{for some \ensuremath{\delta>0}, if \ensuremath{n=2}.}\nonumber 
\end{align}
Then \eqref{A.4} holds, hence \eqref{A.2} holds. 
\end{prop}

\begin{proof}
In view of \eqref{eq: mult}, this amounts to find conditions on $q'>0$
so that 
\[
\int(fg)hd\sigma_{S}<\infty
\]
 when $h\in L^{q'}(S)$ and $fg\in H^{1,p}(M)$. If $n\ge3$, we can
apply the trace theorem, followed by the embedding theorem, 
\begin{equation}
\Tr:H^{1,p}(M)\longrightarrow B_{p,p}^{s}(S)\subset L^{q}(S),\label{A.7}
\end{equation}
with 
\begin{equation}
\begin{aligned}p\end{aligned}
=\frac{n}{n-1},\quad s=1-\frac{1}{p},\quad q=\frac{(n-1)p}{n-1-sp}=\frac{n-1}{n-2},\quad q'=n-1.\label{A.8}
\end{equation}
For $n=2$, \eqref{A.7} applies for all $p\in(1,2)$, again with
$s=1-1/p$, and taking $p\nearrow2$ yields $q\nearrow\infty$, hence
$q'\searrow1$. Thus \eqref{A.5}--\eqref{A.6} imply \eqref{A.4}. 
\end{proof}
\begin{rem}
The trace result is perhaps better known when $S$ is smooth. (Cf.
\cite{BerLof}, Theorem 6.6.1.) However, all the function spaces involved
are invariant under bi-Lipschitz maps. 
\end{rem}

Positive measures satisfying \eqref{A.2} can have much wilder support
than a Lipschitz surface. For example, one can take an infinite sequence
of measures $\mu_{k}$ satisfying the hypotheses of Proposition \ref{P:A.1},
supported on surfaces $S_{k}$, and set $\mu=\sum_{k=1}^{\infty}a_{k}\mu_{k}$,
with positive $a_{k}$ decreasing sufficiently fast.

Here is another class of examples. Let $\mathcal{O}\subset M$ be
an open set whose boundary $\p\mathcal{O}$ is locally the graph of
a continuous function. Then one can take a smooth vector field $X$
on $M$, vanishing nowhere on $\p\mathcal{O}$, whose flow $\mathcal{F}_{X}^{t}$
has the property that, for each $y\in\p\mathcal{O}$, $\mathcal{F}_{X}^{t}y$
belongs to $\mathcal{O}$ for small $t>0$. Then 
\begin{equation}
\mu=X\chi_{\mathcal{O}}\label{A.9}
\end{equation}
is a positive measure, supported by $\p\mathcal{O}$, and it belongs
to $H^{-1,\infty}(M)$. The positivity of $\mu$ is a consequence
of the fact that 
\begin{equation}
\chi_{\mathcal{O}}\circ\mathcal{F}_{X}^{t}\ge\chi_{\mathcal{O}},\label{A.10}
\end{equation}
for all small $t>0$, since 
\begin{equation}
t^{-1}(\chi_{\mathcal{O}}\circ\mathcal{F}_{X}^{t}-\chi_{\mathcal{O}})\longrightarrow X\chi_{\mathcal{O}}\ \text{ in }\ \mathcal{D}'(M),\ \text{ as }\ t\rightarrow0.\label{A.11}
\end{equation}

In the last class of examples, \eqref{A.9}, the support of $\mu$
has topological dimension $n-1$, but its Hausdorff dimension can
be $>n-1$. 

\subsection{Measures with fractal support}

We next produce measures satisfying \eqref{A.4} and supported on
``fractal'' sets of Hausdorff dimension $<n-1$. We make use of
the following result, contained in Theorem 4.7.4 of \cite{Zie89}.
Here, $B_{r}(x)$ denotes the ball of radius $r$ centered at $x$.
\begin{lem}
Let $\mu$ be a positive measure on $M$ with the property that there
exist $A<\infty$ and $\epsilon>0$ such that 
\begin{equation}
\mu(B_{r}(x))\le Ar^{n-q+\epsilon},\quad\forall\,r\in(0,1],\ x\in M.\label{A.12}
\end{equation}
Assume $q\in(1,n)$. Then 
\begin{equation}
\mu\in H^{-1,p}(M),\quad p=q'.\label{A.13}
\end{equation}
\end{lem}

We then see that \eqref{A.4} holds whenever 
\begin{equation}
\mu(B_{r}(x))\le Ar^{\alpha},\quad\alpha>n-1-\frac{1}{n-1}.\label{A.14}
\end{equation}
In particular, for $n=2$, it suffices to have \eqref{A.14} for some
$\alpha>0$.

We will give some explicit examples of a compactly supported measure
on $\mathbb{R}^{2}$ satisfying \eqref{A.14}. It will be clear that
many other examples can be constructed. We start with the Cantor middle
third set $\mathcal{K}\subset[0,1]$. Now put $[0,1]\subset\mathbb{R}\subset\mathbb{R}^{2}$,
say as part of the $x$-axis, so now $\mathcal{K}\subset\mathbb{R}^{2}$.
As is well known (cf.~\cite{Tay06}, p.~170), there is the $\alpha$-dimensional
Hausdorff measure computation 
\begin{equation}
\mathcal{H}^{\alpha}(\mathcal{K})=\gamma_{\alpha},\quad\text{for }\ \alpha=\frac{\log2}{\log3}\approx0.6309,\label{A.15}
\end{equation}
with $0<\gamma_{\alpha}<\infty$ (in fact, $\gamma_{\alpha}=\pi^{\alpha/2}2^{-\alpha}/\Gamma(\alpha/2+1)$).
Set 
\begin{equation}
\mu=\mathcal{H}^{\alpha}\lfloor\mathcal{K},\label{A.16}
\end{equation}
i.e., $\mu(S)=\mathcal{H}^{\alpha}(\mathcal{K}\cap S)$, for Borel
sets $S\subset\mathbb{R}^{2}$. The self similarity of $\mathcal{K}$
enables one to show that 
\begin{equation}
\mu(B_{3^{-k}}(x))\le C\,2^{-k},\label{A.17}
\end{equation}
which readily leads to \eqref{A.14}, with $\alpha$ as in \eqref{A.15}.

The Cantor middle third set described above is one of a family of
Cantor sets $\mathcal{K}(\vartheta)\subset[0,1]$, defined for $\vartheta\in(0,1)$
as follows. Remove from $[0,1]=I$ the open interval of length $\vartheta\ell(I)$,
with the same center as $I$, and repeat this process with the other
closed subintervals that remain. (Thus $\mathcal{K}=\mathcal{K}(1/3)$.)
This time (cf.~\cite{Tay06}, p.~171), one has 
\begin{equation}
\mathcal{H}^{\alpha}(\mathcal{K}(\vartheta))=\gamma_{\alpha},\quad\alpha=\frac{\log2}{\log b},\quad b=\frac{2}{1-\vartheta},\label{A.18}
\end{equation}
and again self-similarity yields 
\begin{equation}
\mu(B_{r}(x))\le Ar^{\alpha},\label{A.19}
\end{equation}
with $\alpha$ as in \eqref{A.18}, when 
\begin{equation}
\mu=\mathcal{H}^{\alpha}\lfloor\mathcal{K}(\vartheta).\label{A.20}
\end{equation}
Note that 
\begin{equation}
\vartheta\searrow0\Rightarrow\alpha\nearrow1,\quad\vartheta\nearrow1\Rightarrow\alpha\searrow0.\label{A.21}
\end{equation}
As before, we put $\mathcal{K}(\vartheta)\subset[0,1]\subset\RR\subset\RR^{2}$,
and regard $\mu$ in \eqref{A.20} as a compactly supported measure
on $\mathbb{R}^{2}$. Thus the push-forward of $\mu$ to a measure
on a compact two-dimensional manifold $M$, via a locally bi-Lipschitz
map, yields a measure on $M$ satisfying \eqref{A.4}, hence \eqref{A.2},
whenever $0<\vartheta<1$.

One way to get measures on higher dimensional spaces satisfying \eqref{A.14}
is to take products. Say $n=n_{1}+n_{2}$ and $\mu_{j}$ are compactly
supported measures on $\mathbb{R}^{n_{j}}$ satisfying 
\begin{equation}
\mu_{j}(B_{r}(x_{j}))\le C_{j}r^{\alpha_{j}},\quad j=1,2,\quad x_{j}\in\mathbb{R}^{n_{j}}.\label{A.22}
\end{equation}
If $x=(x_{1},x_{2})\in\mathbb{R}^{n}$, note that $B_{r}(x)\subset B_{r}(x_{1})\times B_{r}(x_{2})$,
so if 
\begin{equation}
\mu=\mu_{1}\times\mu_{2}\label{A.23}
\end{equation}
is the product measure on $\mathbb{R}^{n}$, we have 
\begin{equation}
\mu(B_{r}(x))\le\mu_{1}(B_{r}(x_{1}))\mu_{2}(B_{r}(x_{2}))\le C_{1}C_{2}r^{\alpha_{1}+\alpha_{2}},\quad x\in\mathbb{R}^{n}.\label{A.24}
\end{equation}
If $\alpha=\alpha_{1}+\alpha_{2}$ satisfies the condition on $\alpha$
in \eqref{A.14}, we get a compactly supported measure on $\mathbb{R}^{n}$
whose push-forward to a measure on an $n$-dimensional compact manifold
$M$, via a locally bi-Lipschitz map, satisfies \eqref{A.14}, hence
\eqref{A.4}, hence \eqref{A.2}.

For example, take $\vartheta_{j}\in(0,1)$, and set 
\begin{equation}
\mu_{j}=\mathcal{H}^{\alpha_{j}}\lfloor\mathcal{K}(\vartheta_{j}),\label{A.25}
\end{equation}
with $\alpha_{j}$ as in \eqref{A.18}, i.e., 
\begin{equation}
\alpha_{j}=\frac{\log2}{\log b_{j}},\quad b_{j}=\frac{2}{1-\vartheta_{j}}.\label{A.26}
\end{equation}
We regard $\mu_{1}$ as a measure on $\mathbb{R}^{2}$ and $\mu_{2}$
as a measure on $\mathbb{R}$, via $\mathcal{K}(\vartheta_{1})\subset[0,1]\subset\mathbb{R}\subset\mathbb{R}^{2}$
and $\mathcal{K}(\vartheta_{2})\subset[0,1]\subset\mathbb{R}$. Thus
$\mu=\mu_{1}\times\mu_{2}$ is a compactly supported measure on $\mathbb{R}^{3}$
(actually supported on a 2D linear subspace of $\mathbb{R}^{3}$).
In this case, the condition for \eqref{A.18} to hold is 
\begin{equation}
\alpha_{1}+\alpha_{2}>\frac{3}{2}.\label{A.27}
\end{equation}
Looking at \eqref{A.15}, we see that \eqref{A.27} fails when $\vartheta_{1}=\vartheta_{2}=1/3$.
In case $\vartheta_{1}=\vartheta_{2}=\vartheta$, the condition that
\eqref{A.27} hold is that 
\begin{equation}
2\,\frac{\log2}{\log2-\log(1-\vartheta)}>\frac{3}{2},\label{A.28}
\end{equation}
or equivalently, 
\begin{equation}
\vartheta<1-2^{-1/3}\approx0.2063.\label{A.29}
\end{equation}

\begin{rem}
When the measures $\mu_{j}$ are given by \eqref{A.25}, $\mu=\mu_{1}\times\mu_{2}$
is supported on $\mathcal{K}(\vartheta_{1})\times\mathcal{K}(\vartheta_{2})$,
but it is generally \textit{not} $(\alpha_{1}+\alpha_{2})$-dimensional
Hausdorff measure on this set. See pp.~70--74 of \cite{Fal85_fract}
for a discussion of this matter. 
\end{rem}

Returning to generalities, we mention that while \eqref{A.12} is
a sufficient condition for \eqref{A.13}, it is not quite necessary.
There is a (somewhat more elaborate) necessary and sufficient condition
for \eqref{A.13} to hold, provided $1<q<n$, given in terms of an
estimate on $\mu(B_{r}(x))$. See Theorem 4.7.5 of \cite{Zie89}.
The condition is subtly weaker than \eqref{A.12}.

\section{Proof of theorem \ref{T:m-damp-1}}

As usual, when dealing with the wave equation, it behooves us to write
the second order (in time) scalar equation as a first order system.
Recall the vector space $\HH=L^{2}(\Omega)\oplus H_{0}^{1}(\Omega)$.
The inner product on $\HH$ is defined as follows: if $U=(v,w)$ and
$F=(f,g)$ are in $\HH$, then 
\begin{align*}
\lll U,F\rrr_{\HH} & =\lll\begin{pmatrix}v\\
w
\end{pmatrix},\begin{pmatrix}f\\
g
\end{pmatrix}\rrr_{\HH}\\
 & =\lll v,f\rrr_{L^{2}(\Omega)}+\lll\nabla w,\nabla g\rrr_{L^{2}(\Omega)}.
\end{align*}
Again, since we are using Dirichlet boundary conditions, the homogeneous
$H_{0}^{1}$ inner product induces a norm, so the inner product on
$\HH$ induces a norm on $\HH$.

Now suppose $u\in C((0,\infty)_{t},H_{0}^{1}(\Omega))$ solves the
damped wave equation \eqref{E:sing-damp}, and set 
\[
U=\begin{pmatrix}u_{t}\\
u
\end{pmatrix}.
\]
Differentiating, we have 
\begin{align*}
U_{t} & =\begin{pmatrix}u_{tt}\\
u_{t}
\end{pmatrix}\\
 & =\begin{pmatrix}-Mu_{t}+\Delta_{g}u\\
u_{t}
\end{pmatrix}\\
 & =\begin{pmatrix}-M & \Delta_{g}\\
I & 0
\end{pmatrix}\begin{pmatrix}u_{t}\\
u
\end{pmatrix}=GU,
\end{align*}
where $G$ is the matrix operator 
\begin{equation}
G=\begin{pmatrix}-M & \Delta_{g}\\
I & 0
\end{pmatrix}.\label{E:G-def}
\end{equation}

Let us compute the weak formulation of the equation $U_{t}=GU$. If
$\Psi=(\phi,\psi)\in\HH\cap D(G)$, where $D(G)$ is the domain of
$G$ (to be computed later), then 
\begin{align*}
\lll U_{t},\Psi\rrr_{\HH} & =\lll\begin{pmatrix}u_{tt}\\
u_{t}
\end{pmatrix},\begin{pmatrix}\phi\\
\psi
\end{pmatrix}\rrr_{\HH}\\
 & =\lll\begin{pmatrix}-Mu_{t}+\Delta_{g}u\\
u_{t}
\end{pmatrix},\begin{pmatrix}\phi\\
\psi
\end{pmatrix}\rrr_{\HH}\\
 & =-\lll Mu_{t},{\phi}\rrr-\lll\nabla u,\nabla\phi\rrr_{L^{2}(\Omega)}+\lll\nabla u_{t},\nabla\psi\rrr_{L^{2}(\Omega)}.
\end{align*}

The left hand side is similarly computed 
\begin{align*}
\lll\begin{pmatrix}u_{tt}\\
u_{t}
\end{pmatrix},\begin{pmatrix}\phi\\
\psi
\end{pmatrix}\rrr & =\lll u_{tt},\phi\rrr_{L^{2}(\Omega)}+\lll\nabla u_{t},\nabla\psi\rrr_{L^{2}(\Omega)}.
\end{align*}
Making the obvious cancellations, we have 
\[
\lll u_{tt},\phi\rrr_{L^{2}(\Omega)}=-\lll Mu_{t},\phi\rrr-\lll\nabla u,\nabla\phi\rrr_{L^{2}(\Omega)},
\]
which is the same weak formulation as the scalar equation. Putting
the other terms back in, we have the weak formulation in terms of
the system: 
\[
\lll U_{t},\Psi\rrr_{\HH}=\lll\begin{pmatrix}-\nabla u\\
u_{t}
\end{pmatrix},\begin{pmatrix}\nabla\phi\\
\psi
\end{pmatrix}\rrr_{\HH}-\lll Mu_{t},\phi\rrr.
\]

We begin by computing the adjoint of the matrix operator $G$. Let
\[
U=\begin{pmatrix}v\\
w
\end{pmatrix},F=\begin{pmatrix}f\\
g
\end{pmatrix}\in\HH\cap D(G),
\]
where the domain $D(G)$ will be determined shortly, and compute 
\begin{align*}
\lll GU,F\rrr_{\HH} & =\lll\begin{pmatrix}-M & \Delta_{g}\\
I & 0
\end{pmatrix}\begin{pmatrix}v\\
w
\end{pmatrix},\begin{pmatrix}f\\
g
\end{pmatrix}\rrr_{\HH}\\
 & =\lll\begin{pmatrix}-Mv+\Delta w\\
v
\end{pmatrix},\begin{pmatrix}f\\
g
\end{pmatrix}\rrr_{\HH}\\
 & =-\lll Mv,f\rrr+\lll\Delta w,f\rrr_{L^{2}(\Omega)}+\lll\nabla v,\nabla g\rrr_{L^{2}(\Omega)}\\
 & =\lll\begin{pmatrix}v\\
w
\end{pmatrix},\begin{pmatrix}-Mf-\Delta_{g}g\\
-f
\end{pmatrix}\rrr_{\HH}\\
 & =\lll\begin{pmatrix}v\\
w
\end{pmatrix},\begin{pmatrix}-M & -\Delta_{g}\\
-I & 0
\end{pmatrix}\begin{pmatrix}f\\
g
\end{pmatrix}\rrr_{\HH},
\end{align*}
so that the adjoint is 
\[
G^{*}=\begin{pmatrix}-M & -\Delta_{g}\\
-I & 0
\end{pmatrix}.
\]
Since $G\neq G^{*}$, we will consider only the real part to show
that $G$ is dissipative. That is, we need to show $\Re\lll GU,U\rrr\leq0$.
Let 
\[
U=\begin{pmatrix}v\\
w
\end{pmatrix}\in\HH\cap D(G).
\]
We compute 
\begin{align*}
\Re\lll GU,U\rrr_{\HH} & =\Re\lll\begin{pmatrix}-M & \Delta_{g}\\
I & 0
\end{pmatrix}\begin{pmatrix}v\\
w
\end{pmatrix},\begin{pmatrix}v\\
w
\end{pmatrix}\rrr_{\HH}\\
 & =\Re\lll\begin{pmatrix}-Mv+\Delta w\\
v
\end{pmatrix},\begin{pmatrix}v\\
w
\end{pmatrix}\rrr_{\HH}\\
 & =-\lll Mv,v\rrr+\Re\lll\Delta_{g}w,v\rrr_{L^{2}(\Omega)}+\Re\lll\nabla v,\nabla w\rrr_{L^{2}(\Omega)}\\
 & =-\lll Mv,v\rrr-\Re\lll\nabla w,\nabla v\rrr_{L^{2}(\Omega)}+\Re\lll\nabla v,\nabla w\rrr_{L^{2}(\Omega)}\\
 & =-\lll Mv,v\rrr\\
 & \leq0,
\end{align*}
since $M$ is symmetric and non-negative according to Assumption 1.
This shows that $G$ is dissipative.

In order to conclude that $G$ is maximal dissipative, we also need
to show that $(I-G)$ is an isomorphism between appropriate spaces.
In particular, we will describe the domain $D(I-G)$ and show that
\[
(I-G):D(I-G)\to\HH
\]
is an isomorphism. We will also conclude that the inverse mapping
$(I-G)^{-1}:\HH\to\HH$ is compact. This is where we use the assumed
compact mapping properties of $M:H_{0}^{1}(\Omega)\to H^{-1}(\Omega)$.

Let 
\[
U=\begin{pmatrix}v\\
w
\end{pmatrix}\in\HH\cap D(G),
\]
and 
\[
F=\begin{pmatrix}f\\
g
\end{pmatrix}\in\HH.
\]
We want to understand the solvability of $(I-G)U=F$ and show that
for any $F\in\HH$, there is a unique $U\in D(I-G)$ solving this
equation. Writing out in components: 
\begin{align*}
(I-G)\begin{pmatrix}v\\
w
\end{pmatrix} & =\begin{pmatrix}(1+M)v-\Delta_{g}w\\
-v+w
\end{pmatrix}\\
 & =\begin{pmatrix}f\\
g
\end{pmatrix},
\end{align*}
or 
\[
\begin{cases}
(1+M)v-\Delta_{g}w=f,\\
-v+w=g.
\end{cases}
\]
Using the second equation to eliminate $v=w-g$ from the first, we
get the scalar equation 
\begin{equation}
(1+M)w-\Delta_{g}w=f+(1+M)g.\label{E:solve-for-w}
\end{equation}
Now $F\in\HH$ means that $f\in L^{2}(\Omega)$ and $g\in H_{0}^{1}(\Omega)$.
The assumed mapping properties of $M$ implies 
\[
Mg\in H^{-1}(\Omega).
\]
Hence the right hand side of \eqref{E:solve-for-w} is in $H^{-1}(\Omega)$
or better. Since $U\in D(I-G)\subset\HH$, we know that $w\in H_{0}^{1}(\Omega)$
or better.

Now observe that $(1-\Delta_{g}):H_{0}^{1}(\Omega)\to H^{-1}(\Omega)$
is an isomorphism by the usual Fredholm theory. Since $M:H_{0}^{1}(\Omega)\to H^{-1}(\Omega)$
is compact, 
\[
(1-\Delta_{g}+M):H_{0}^{1}(\Omega)\to H^{-1}(\Omega)
\]
is a compact perturbation of an isomorphism, hence has Fredholm index
$0$. We therefore know that the codimension of the image is the dimension
of the kernel. We therefore need only establish injectivity to recover
surjectivity. Assume then that $w\in\ker(1-\Delta_{g}+M)$ and compute
\begin{align*}
0 & =\lll(1-\Delta_{g}+M)w,w\rrr_{L^{2}(\Omega)}\\
 & =\|w\|_{L^{2}(\Omega)}^{2}+\|\nabla w\|_{L^{2}(\Omega)}^{2}+\lll Mw,w\rrr\\
 & \geq\|w\|_{L^{2}(\Omega)}^{2}+\|\nabla w\|_{L^{2}(\Omega)}^{2}
\end{align*}
since $M$ is symmetric and non-negative. Hence $w=0$ and the kernel
is trivial as claimed.

We have shown that 
\[
(1-\Delta_{g}+M):H_{0}^{1}(\Omega)\to H^{-1}(\Omega)
\]
is an isomorphism. We will use this, together with the mapping properties
of $M$ to establish that $(I-G)^{-1}$ is compact and also describe
$D(I-G)$. We give the right hand side of \eqref{E:solve-for-w} a
name. Let $\Phi:\HH\to H^{-1}$ be defined by 
\[
\Phi\begin{pmatrix}f\\
g
\end{pmatrix}=f+g+Mg.
\]
Recall that $f\in L^{2}(\Omega)$ and $g\in H_{0}^{1}(\Omega)$. The
embeddings 
\[
L^{2}(\Omega)\hookrightarrow H^{-1}(\Omega),\quad H_{0}^{1}(\Omega)\hookrightarrow H^{-1}(\Omega)
\]
are compact, as is $M:H_{0}^{1}(\Omega)\to H^{-1}(\Omega)$, hence
$\Phi:\HH\to H^{-1}$ is compact. Returning to \eqref{E:solve-for-w},
we know the operator on the left hand side is an isomorphism, hence
invertible as a bounded linear operator. We have 
\[
w=(1-\Delta_{g}+M)^{-1}\Phi\begin{pmatrix}f\\
g
\end{pmatrix}=:\Psi\begin{pmatrix}f\\
g
\end{pmatrix}.
\]
The mapping $\Psi$ defined above is a composition of a bounded linear
operator with a compact operator, taking values in $H_{0}^{1}(\Omega)$.
Hence $\Psi:\HH\to H_{0}^{1}(\Omega)$ is compact. Recalling that
$v=w-g$, we have 
\[
v=\Psi\begin{pmatrix}f\\
g
\end{pmatrix}-g.
\]

The original equation we considered was $(I-G)U=F$. In particular,
the second component in $F$, $g\in H_{0}^{1}(\Omega)$. Hence 
\[
v=\Psi\begin{pmatrix}f\\
g
\end{pmatrix}-g\in H_{0}^{1}(\Omega)
\]
if $F\in\HH$. As the embedding $H_{0}^{1}(\Omega)\hookrightarrow L^{2}(\Omega)$
is compact, the mapping defining $v$ is compact when mapping into
$L^{2}(\Omega)$. Taken together, we write for $U$ so obtained 
\[
U=(I-G)^{-1}F,
\]
where 
\[
(I-G)^{-1}\begin{pmatrix}f\\
g
\end{pmatrix}=\begin{pmatrix}\Psi\begin{pmatrix}f\\
g
\end{pmatrix}-g\\
\Psi\begin{pmatrix}f\\
g
\end{pmatrix}
\end{pmatrix}.
\]
Then as a mapping 
\[
(I-G)^{-1}:\HH\to\HH,
\]
$(I-G)^{-1}$ is compact. The domain $D(I-G)$ is the image of $\HH$
under this transformation.

\section{Decay}

In this section, we examine the general spectral theory of $G$
and the associated semigroup $e^{tG}$ which is the solution operator
for the damped wave system. We have shown that $(I-G)^{-1}:\HH\to\HH$
is compact. It follows that $G$ is closed and $1\notin\Spec G$.
As a classical result (see \cite{Kat95}, Theorem III.6.29) we deduce
that $G$ has compact resolvent with eigenvalues of finite multiplicities,
accumulating only possibly at $\infty\in\overline{\cx}$. As a direct
consequence of $\Re\langle GU,U\rangle_{\HH}\leq0$ for any $U\in\HH$,
the spectrum of $G$ satisfies 
\[
\Spec G\subset\{z\in\cx,\ \Re z\leq0\},
\]
so the associated semigroup 
\[
e^{tG}:D(I-G)\to\HH
\]
is a contraction, in the sense that 
\[
\|e^{tG}\|_{\HH\to\HH}\leq1.
\]

We also deduce that in particular, the image of $\HH$ is $D(I-G)$,
which is a subspace of $H_{0}^{1}\oplus H_{0}^{1}\subset\HH.$ When
operating on eigenfunctions of $G$, the evolution is easy to compute.
If 
\[
F\in D(I-G)
\]
is an eigenfunction of $G$ with eigenvalue $\zeta\in\cx$, we have
$GF=\zeta F$ and then 
\[
e^{tG}F=e^{t\zeta}F.
\]
Hence the decay properties of the semigroup are dictated by the eigenvalues
of $G$. In particular, since $\zeta$ must lie in the closed left
half-plane in $\cx$, any purely imaginary eigenvalues correspond
to non-decaying states. In the absence of any purely imaginary eigenvalues,
the asymptotic distribution of $\Re\zeta$ as $|\Im\zeta|\to\infty$
determines the decay rate (see, for example, \cite{Chr-wave-2}).

Let us fix some notations. If $\zeta\in\Spec(G)$ is an eigenvalue,
set $V_{\zeta}\subset D(G)$ to be the generalized eigenspace corresponding
to $\zeta$, and let $V_{\zeta}^{\#}\subset V_{\zeta}$ be the primitive
eigenspace. That is, $V_{\zeta}^{\#}$ consists of eigenvectors of
$G$ with eigenvalue $\zeta$, and $V_{\zeta}$ consists of eigenvectors
of $G^{k}$, $k\geq1$, with eigenvalue $\zeta$. As we noted above,
each $V_{\zeta}$ and $V_{\zeta}^{\#}$ is finite dimensional.

Let $\zeta\in\Spec G$ so $\Re\zeta\leq0$. Take 
\[
\left(\begin{array}{c}
f\\
g
\end{array}\right)\in V_{\zeta}^{\#}.
\]
Recall that this implies 
\[
g\in H_{0}^{1},\qquad f=\zeta g,\qquad\Delta g=\zeta^{2}g+\zeta Mg.
\]

\begin{prop}
If $(f,g)\in V_{\zeta}^{\#}$ with $\|(f,g)\|_{\mathcal{H}}=1$ then
$\Re\zeta=-\langle Mf,f\rangle.$ 
\end{prop}

In particular, if $M$ is the multiplication by the measure $\mu$,
then 
\[
\Re\zeta=-\int_{\Omega}|f|^{2}d\mu.
\]

\begin{proof}
Note that 
\[
G\left(\begin{array}{c}
f\\
g
\end{array}\right)=\left(\begin{array}{c}
\Delta g-Mf\\
f
\end{array}\right)=\zeta\left(\begin{array}{c}
f\\
g
\end{array}\right),
\]
so 
\[
\left\langle G\left(\begin{array}{c}
f\\
g
\end{array}\right),\left(\begin{array}{c}
f\\
g
\end{array}\right)\right\rangle _{\HH}=\langle\Delta g,f\rangle_{L^{2}(\Omega)}-\langle Mf,f\rangle-\langle\Delta f,g\rangle_{L^{2}(\Omega)}.
\]
Then the above implies 
\begin{align*}
\zeta\left\Vert \left(\begin{array}{c}
f\\
g
\end{array}\right)\right\Vert _{\mathcal{H}}^{2} & =\langle\Delta g,f\rangle_{L^{2}(\Omega)}-\langle Mf,f\rangle-\overline{\langle\zeta f+Mf,f\rangle}_{L^{2}(\Omega)}\\
 & =\zeta\|f\|_{L^{2}}^{2}+\langle Mf,f\rangle-\langle Mf,f\rangle-\overline{\zeta}\|f\|_{L^{2}}^{2}-\langle Mf,f\rangle\\
 & =(\zeta-\overline{\zeta})\|f\|_{L^{2}}^{2}-\langle Mf,f\rangle.
\end{align*}
Taking the real part gives the result. 
\end{proof}

\subsection{Primitive eigenspaces with purely imaginary eigenvalues}

Let us denote the primitive eigenspaces with purely imaginary eigenvalues
by $V_{i\lambda}^{\#}$, where this notation assumes $\lambda\in\mathbb{R}$.
An immediate consequence of the preceding proposition is 
\begin{cor}
\label{cor: imag M}We have 
\[
\begin{pmatrix}v\\
w
\end{pmatrix}\in V_{i\lambda}^{\#}\quad\Leftrightarrow\quad\langle Mv,v\rangle=0.
\]
\end{cor}

In particular, if $M=M_{\mu}$ is the multiplication by a non-negative
measure satisfying the compactness assumptions, then a necessary condition
for $(v,w)^{T}\in V_{i\lambda}^{\#}$ is that 
\[
\lll M_{\mu}v,v\rrr=\int_{\Omega}|v|^{2}d\mu=0.
\]
This is true if and only if $v$ is supported in the $0$-set of the
measure $\mu$. A similar converse holds as well. If $v$ satisfies
$\int_{\Omega}|v|^{2}d\mu=0$, then, we recall that $M_{\mu}v\in H^{-1}(\Omega)$,
so for any $f\in H_{0}^{1}(\Omega)$, we have 
\begin{align*}
|\lll M_{\mu}v,f\rrr| & =\left|\int_{\Omega}v\bar{f}d\mu\right|\\
 & \leq\left(\int_{\Omega}|v|^{2}d\mu\right)^{1/2}\left(\int_{\Omega}|f|^{2}d\mu\right)^{1/2}.
\end{align*}
The latter integral is bounded, so the use of Hölder's inequality
is justified; the first integral is $0$ by assumption. Hence we have
that $M_{\mu}v=0$ in $H^{-1}(\Omega)$ (that is, in the weak sense). 

Returning to the matrix 
\[
G:D(G)\to H^{-1}(\Omega)\oplus H_{0}^{1}(\Omega),
\]
we have if $(v,w)^{T}\in V_{i\lambda}^{\#}$, then 
\begin{align*}
G\begin{pmatrix}v\\
w
\end{pmatrix} & =\begin{pmatrix}-M_{\mu}v+\Delta w\\
v
\end{pmatrix}\\
 & =\begin{pmatrix}\Delta w\\
v
\end{pmatrix}\\
 & =\begin{pmatrix}0 & \Delta\\
I & 0
\end{pmatrix}\begin{pmatrix}v\\
w
\end{pmatrix}.
\end{align*}
But 
\[
G\begin{pmatrix}v\\
w
\end{pmatrix}=i\lambda\begin{pmatrix}v\\
w
\end{pmatrix},
\]
so 
\[
\begin{cases}
\Delta w=i\lambda v\\
v=i\lambda w,
\end{cases}
\]
Substituting, we have $-\Delta w=\lambda^{2}w$, which means that
$w$ is a Dirichlet eigenfunction on $\Omega$. Hence if $w$ is such
an eigenfunction and $v=i\lambda w$, then the solution $u$ to the
damped wave equation 
\[
\begin{cases}
(\p_{t}^{2}-\Delta+M_{\mu}\p_{t})u=0,\text{ on }(0,\infty)_{t}\times\Omega,\\
u|_{\p\Omega}=0,\text{ on }[0,\infty)_{t}\times\p\Omega,\\
u(0,x)=w(x),\,\,u_{t}(0,x)=v(x),
\end{cases}
\]
satisfies the usual undamped wave equation: 
\[
\begin{cases}
(\p_{t}^{2}-\Delta)u=0,\text{ on }(0,\infty)_{t}\times\Omega,\\
u|_{\p\Omega}=0,\text{ on }[0,\infty)_{t}\times\p\Omega,\\
u(0,x)=w(x),\,\,u_{t}(0,x)=v(x).
\end{cases}
\]


Next, let 
\[
\MM=\overline{\bigoplus_{\lambda\in\reals}V_{i\lambda}^{\#}}
\]
denote the $\HH$ closure of all of the eigenspaces with purely imaginary
eigenvalue. Since the spectrum is pure point, this is, of course,
a sum over at most countably many such $\lambda\in\reals$. Let $\MM^{\perp}\subset\HH$
be the orthogonal complement to $\MM$. The following proposition
is the central result of this paper, it will readily imply Theorem
\ref{thm: stabilization}. 
\begin{prop}
\label{prop: decay}Let $G$, $\MM$, and $\MM^{\perp}$ be defined
as above.
\begin{enumerate}
\item For each $t$, the semigroup $e^{tG}:\MM\to\MM$ is unitary. 
\item The space $\MM^{\perp}$ is also invariant under $e^{tG}$. 
\item The resolvent $(I-G)^{-1}:\MM^{\perp}\to\MM^{\perp}$ is compact. 
\item If $F\in\MM^{\perp}$, then 
\[
\|e^{tG}F\|_{\HH}\to0,
\]
as $t\to\infty$. 
\end{enumerate}
\end{prop}

\begin{proof}
Part 1 follows directly from the definition of $\MM$. To prove part
2, it suffices to show that no generalized eigenvector of $G$ with
imaginary eigenvalue is in $\MM^{\perp}$. Suppose for some $\lambda\in\reals$
that 
\[
\MM^{\perp}\cap V_{i\lambda}\neq\emptyset.
\]
Since $V_{i\lambda}$ is finite dimensional, the vector space $V=\MM^{\perp}\cap V_{i\lambda}$
is finite dimensional, and consists of generalized eigenvectors of
$G$. That means 
\[
G|_{V}=A
\]
is a finite square matrix $A$. Putting $A$ into Jordan canonical
form shows us that there must be at least one honest eigenvector of
$G$ in $V$, still with eigenvalue $i\lambda$. Hence $V\cap V_{i\lambda}^{\#}\neq\emptyset$,
which is a contradiction. This shows that for every $i\lambda\in\sigma(G)$,
$V_{i\lambda}\subset\MM$, which proves 2. To prove part 3, we write
the resolvent as an integral: 
\[
(I-G)^{-1}=\int_{0}^{\infty}e^{-t}e^{tG}dt.
\]
Since $e^{tG}$ is a contraction semigroup, the integral converges
in operator norm, and indeed we can formally integrate by parts to
derive this formula. Since $e^{tG}:\MM^{\perp}\to\MM^{\perp}$, we
have that the resolvent 
\[
(I-G)^{-1}:\MM^{\perp}\to\MM^{\perp}
\]
as well, and is compact, since it is compact on the larger space $\HH$.

For the last part, we use the following abstract decay result \cite{BRT-D-damp}:
if $A$ is a maximal dissipative operator on a Hilbert space such
that $A$ has compact resolvent and no purely imaginary eigenvalues,
then 
\[
\lim_{t\to+\infty}e^{tA}=0,\quad\mbox{strongly.}
\]

The conclusion of 4 follows by considering $e^{tG}:\MM^{\perp}\to\MM^{\perp}$,
since we have seen above that $G|_{\MM^{\perp}}$ has no purely imaginary
eigenvalues. 
\end{proof}
As an immediate corollary of the above result, note that $\lim_{t\to+\infty}e^{tA}=0$
strongly if and only if $\MM=\emptyset$. To get Theorem \ref{thm: stabilization},
let $M=M_{\mu}$ as in Section \ref{S:compact-measure-mapping}. From
the discussion following Corollary \ref{cor: imag M}, $\MM\neq\emptyset$
if and only if there is a Dirichlet eigenfunction $\psi_{\lambda}$
of $-\Delta_{g}$ such that 
\[
\int|\psi_{\lambda}|d\mu=0,
\]
in which case $\psi_{\lambda}$ is a stationary solution of \eqref{E:sing-damp}
and $\psi_{\lambda}\in\MM$. 

\subsection{Spectral gap and exponential decay}

We say that the semigroup $(e^{tG})_{t\geq0}$ decays exponentially
if there is $C,\beta>0$ such that

\[
\|e^{tG}\|_{\mathcal{L}(\mathcal{H})}\leq Ce^{-\beta t},\qquad\forall t\geq0.
\]
Hence uniform stabilization is equivalent to an exponential decay
of the semigroup. On the other hand, we say that $G$ has a spectral
gap if there is $A>0$ such that $\forall\zeta\in\Spec G$, 
\[
|\Re\zeta|\geq A.
\]
Exponential decay of the semigroup and spectral gap are not necessarily
equivalent in non-selfadjoint spectral problems, due to pseudospectral
phenomena. However we have the following result : 
\begin{thm}
\label{thm: Gap} Let $A>0$ and consider the strip $\mathcal{S}=\{z\in\mathbb{C}:\Re z\in[-A,0]\}$. 

(i) If the operator $G$ has no spectral gap, then
\[
\|e^{tG}\|_{\mathcal{L}(\mathcal{H})}=1,\quad\forall t\geq0.
\]

(ii) Suppose that for some $A>0$ the strip $\mathcal{S}$ is in the
resolvent set of $G$, and that 
\[
\sup_{z\in\mathcal{S}}\|(z-G)^{-1}\|<\infty.
\]
Then the semigroup $(e^{tG})_{t\geq0}$ decays exponentially. 
\end{thm}

\begin{proof}
It is clear that if there exists $(\zeta_{j})_{j\in\mathbb{N}}$ with
$\zeta_{j}\in\Spec G$ and $\Re\zeta_{j}\xrightarrow{j\to\infty}0$,
then $\|e^{tG}\|_{\mathcal{L}(\mathcal{H})}=1,\quad\forall t\geq0.$ 

Conversely, let us assume that there is $A>0$ such that $\forall\zeta\in\Spec G$,
$|\Re\zeta|\geq A.$ The preimage of $\mathcal{C}=\{z\in\mathbb{\mathbb{C}}:e^{-A}\leq|z|\leq1\}$
by the exponential map is precisely $\mathcal{S}$, and the uniform
boundedness of the resolvent on $\mathcal{S}$ allows us to apply
a theorem of Prüss \cite{Pru84} to conclude that $\C$ is in the
resolvent set of $e^{G}$. Hence the spectral radius of $e^{G}$ is
$\leq e^{-A}$, namely 
\[
\limsup_{k\to\infty}\|e^{kG}\|_{\mathcal{L}(\mathcal{H})}^{\frac{1}{k}}\leq e^{-A}<1.
\]
This implies that $\|e^{tG}\|_{\mathcal{L}(\mathcal{H})}<1$ for some
$t>0$ and concludes the proof of the theorem by using the semigroup
property.
\end{proof}
\bibliographystyle{alpha}
\bibliography{sing-damp_3-bib}

\def\cprime{$'$} \def\cftil#1{\ifmmode\setbox7\hbox{$\accent"5E#1$}\else
  \setbox7\hbox{\accent"5E#1}\penalty 10000\relax\fi\raise 1\ht7
  \hbox{\lower1.15ex\hbox to 1\wd7{\hss\accent"7E\hss}}\penalty 10000
  \hskip-1\wd7\penalty 10000\box7}
\begin{thebibliography}{BRT82}

\bibitem[BL76]{BerLof}
J\"oran Bergh and J\"orgen L\"ofstr\"om.
\newblock {\em Interpolation spaces. {A}n introduction}, volume No. 223 of {\em
  Grundlehren der Mathematischen Wissenschaften}.
\newblock Springer-Verlag, Berlin-New York, 1976.

\bibitem[BLR92]{BLR92}
Claude Bardos, Gilles Lebeau, and Jeffrey Rauch.
\newblock Sharp sufficient conditions for the observation, control, and
  stabilization of waves from the boundary.
\newblock {\em SIAM J. Control Optim.}, 30(5):1024--1065, 1992.

\bibitem[BRT82]{BRT-D-damp}
A.~Bamberger, J.~Rauch, and M.~Taylor.
\newblock A model for harmonics on stringed instruments.
\newblock {\em Arch. Rational Mech. Anal.}, 79(4):267--290, 1982.

\bibitem[Chr09]{Chr-wave-2}
Hans Christianson.
\newblock Applications of cutoff resolvent estimates to the wave equation.
\newblock {\em Math. Res. Lett.}, 16(4):577--590, 2009.

\bibitem[Fal86]{Fal85_fract}
K.~J. Falconer.
\newblock {\em The geometry of fractal sets}, volume~85 of {\em Cambridge
  Tracts in Mathematics}.
\newblock Cambridge University Press, Cambridge, 1986.

\bibitem[JTZ98]{JTZ-D-damp}
St{\'e}phane Jaffard, Marius Tucsnak, and Enrique Zuazua.
\newblock Singular internal stabilization of the wave equation.
\newblock {\em J. Differential Equations}, 145(1):184--215, 1998.

\bibitem[Kat95]{Kat95}
Tosio Kato.
\newblock {\em Perturbation theory for linear operators}.
\newblock Classics in Mathematics. Springer-Verlag, Berlin, 1995.
\newblock Reprint of the 1980 edition.

\bibitem[MS09]{MazSha09}
Vladimir~G. Maz'ya and Tatyana~O. Shaposhnikova.
\newblock {\em Theory of {S}obolev multipliers}, volume 337 of {\em Grundlehren
  der mathematischen Wissenschaften [Fundamental Principles of Mathematical
  Sciences]}.
\newblock Springer-Verlag, Berlin, 2009.
\newblock With applications to differential and integral operators.

\bibitem[Pr{\"u}84]{Pru84}
Jan Pr{\"u}ss.
\newblock On the spectrum of {$C\sb{0}$}-semigroups.
\newblock {\em Trans. Amer. Math. Soc.}, 284(2):847--857, 1984.

\bibitem[Ral69]{Ral69_loc}
James~V. Ralston.
\newblock Solutions of the wave equation with localized energy.
\newblock {\em Comm. Pure Appl. Math.}, 22:807--823, 1969.

\bibitem[RS96]{RunSic96}
Thomas Runst and Winfried Sickel.
\newblock {\em Sobolev spaces of fractional order, {N}emytskij operators, and
  nonlinear partial differential equations}, volume~3 of {\em De Gruyter Series
  in Nonlinear Analysis and Applications}.
\newblock Walter de Gruyter \& Co., Berlin, 1996.

\bibitem[RT74]{RaTa-74}
Jeffrey Rauch and Michael Taylor.
\newblock Exponential decay of solutions to hyperbolic equations in bounded
  domains.
\newblock {\em Indiana Univ. Math. J.}, 24:79--86, 1974.

\bibitem[Tay06]{Tay06}
Michael~E. Taylor.
\newblock {\em Measure theory and integration}, volume~76 of {\em Graduate
  Studies in Mathematics}.
\newblock American Mathematical Society, Providence, RI, 2006.

\bibitem[Zie89]{Zie89}
William~P. Ziemer.
\newblock {\em Weakly differentiable functions}, volume 120 of {\em Graduate
  Texts in Mathematics}.
\newblock Springer-Verlag, New York, 1989.
\newblock Sobolev spaces and functions of bounded variation.

\end{thebibliography}
 
\end{document}